\documentclass[12pt]{article}
\usepackage{latexsym,amsfonts,amsmath,graphics,amsthm}
\usepackage{epsfig}
\usepackage{mathrsfs}

\newtheorem{theorem}{Theorem}
\newtheorem{lemma}{Lemma}

\newtheorem{definition}{Definition}

\newtheorem{remark}{Remark}

\newcommand{\bsk}{\boldsymbol{k}}

\newcommand{\bsx}{\boldsymbol{x}}

\newcommand{\bsell}{\boldsymbol{\ell}}

\newcommand{\bsy}{\boldsymbol{y}}

\newcommand{\icomp}{\mathtt{i}}
\newcommand{\bszero}{\boldsymbol{0}}

\newcommand{\rd}{\,\mathrm{d}}

\newcommand{\NN}{\mathbb{N}}
\newcommand{\ZZ}{\mathbb{Z}}
\newcommand{\QQ}{\mathbb{Q}}
\newcommand{\CC}{\mathbb{C}}

\newcommand{\cS}{\mathcal{S}}
\newcommand{\FF}{\mathbb{F}}

\newcommand{\RR}{{\mathbb R}}

\newcommand{\len}{{\rm len}_q}
\newcommand{\walb}{\,_q{\rm wal}}

\makeatletter
\newcommand{\rdots}{\mathinner{\mkern1mu\lower-1\p@\vbox{\kern7\p@\hbox{.}}
\mkern2mu \raise4\p@\hbox{.}\mkern2mu\raise7\p@\hbox{.}\mkern1mu}}
\makeatother

\allowdisplaybreaks

\begin{document}

\title{A metrical lower bound on the star discrepancy of digital sequences}
\author{Gerhard Larcher\thanks{G. L. is partially supported by the Austrian Science Foundation (FWF), project P21943-N18} and Friedrich Pillichshammer}
\date{}
\maketitle

\begin{abstract}
In this paper we study uniform distribution properties of digital sequences over a finite field of prime order. In 1998 it was shown by Larcher that for almost all $s$-dimensional digital sequences the star discrepancy $D_N^\ast$ satisfies an upper bound of the form $D_N^\ast=O((\log N)^s (\log \log N)^{2+\varepsilon})$ for any $\varepsilon>0$. Generally speaking it is much more difficult to obtain good lower bounds for specific sequences than upper bounds. Here we show that Larchers result is best possible up to some $\log \log N$ term. More detailed, we prove that for almost all $s$-dimensional digital sequences the star discrepancy satisfies $D_N^\ast \ge c(q,s) (\log N)^s \log \log N$ for infinitely many $N \in \NN$, where $c(q,s)>0$ only depends on $q$ and $s$ but not on $N$.   
\end{abstract}


\section{Introduction and statement of the main result}

In this paper we study uniform distribution properties of infinite sequences in the multi-dimensional unit cube. Sequences with excellent distribution properties are required as underlying nodes in so-called quasi-Monte Carlo algorithms for multivariate integration, see for example, \cite{DP10,kuinie,niesiam} for more information in this direction.
 
Let $\cS=(\bsy_n)_{n \ge 0}$ be an infinite sequence in the $s$-dimensional unit cube $[0,1)^s$. For $\bsx=(x_1,\ldots,x_s) \in [0,1]^s$ and $N \in \NN$ (by $\NN$ we denote the set of positive integers) the {\it local discrepancy} $D(\bsx,N)$ of $\cS$ is the difference between the number of indices $n=0,\ldots,N-1$ for which $\bsy_n$ belongs to the interval $[\bszero, \bsx)=\prod_{j=1}^s [0,x_j)$ and the expected number $N x_1\cdots x_s$ of points in $[\bszero,\bsx)$ if we assume a perfect uniform distribution on $[0,1]^s$, i.e., $$D(\bsx,N)=\#\{0 \le n < N \ : \ \bsx_n \in [\bszero,\bsx)\}-N x_1 \cdots x_s.$$

A sequence $\cS$ is called {\it uniformly distributed} if and only if for any $\bsx \in [0,1]^s$ the normalized local discrepancy $D(\bsx,N)/N$ tends to 0 for growing $N$. The {\it star discrepancy} $D_N^\ast$ of a sequence $\cS$ is the $L_\infty$-norm of the local discrepancy, i.e., $$D_N^\ast(\cS)=\|D(\bsx,N)\|_{\infty}.$$ 

It follows from a result of Roth~\cite{roth} that there exists a quantity $c(s)>0$ such that for any sequence $\cS$ in $[0,1)^s$ we have 
\begin{equation}\label{bdroth}
 D_N^{\ast}(\cS) \ge c(s) (\log N)^{s/2} \ \ \ \mbox{ for infinitely many }\ \ N \in \NN.
\end{equation}
For a proof, see, for example, \cite[Chapter~2, Theorem~2.2]{kuinie}. However, the exact lower order of star discrepancy in $N$ is still one of the most famous open problems in the theory of uniform distribution. Many people believe that there exists some $c(s) >0$ such that for any infinite sequence $\cS$ in $[0,1)^s$ we have 
\begin{equation}\label{conj}
D_N^\ast(\cS) \ge c(s) (\log N)^s \ \ \ \mbox{ for infinitely many }\ \ N \in \NN. 
\end{equation}
Until now it is not known whether this conjecture is true for any $s \ge 2$. For dimension $s=1$ the correctness of \eqref{conj}  has been shown by Schmidt~\cite{Schm72distrib}. If the conjectured lower bound \eqref{conj} is correct, then it would be best possible in the order of magnitude in $N$. An excellent introduction to this topic can be found in the book of Kuipers and Niederreiter \cite{kuinie}. Further books dealing with uniform distribution theory, discrepancy and applications are \cite{BC,DP10,dt,matou,niesiam} which can also be warmly recommended. \\

One example of sequences which can achieve a discrepancy bound of order of magnitude $(\log N)^s$ in $N$ are so-called digital sequences over a finite field $\FF_q$ of prime-power order $q$. Such sequences go back to Sobol' \cite{sob} and Faure \cite{fau}, but the detailed introduction and investigation of the general concept was first given by Niederreiter in \cite{nie87}. In the following we always assume that $q$ is a prime number. Hence we can identify $\FF_q$ with $\ZZ_q=\{0,1,\ldots,q-1\}$ equipped with arithmetic operations modulo $q$. By $(\ZZ_q^{\NN})^\top$ we will denote the set of infinite dimensional column vectors over $\ZZ_q$.

\begin{definition}[digital sequences]\rm
Let $s \in \NN$ and $q$ be a prime number. Let $C_1,\ldots, C_s \in \ZZ_q^{\NN \times \NN}$ be $\NN \times \NN$ matrices over $\ZZ_q$. Let $n \in \NN_0$, where $\NN_0=\NN \cup \{0\}$, with $q$-adic expansion $n = n_0 + n_1 q + n_2 q^2+\cdots $ (this expansion is obviously finite) and set $$\vec{n} = (n_0, n_1, n_2, \ldots )^\top \in (\ZZ_q^{\mathbb{N}})^\top.$$ Then define
\begin{equation*}
\vec{x}_{n,j} = C_j \vec{n} \quad \mbox{for } j = 1,\ldots, s
\end{equation*}
where all arithmetic operations are taken modulo $q$. Let $\vec{x}_{n,j} = (x_{n,j,1}, x_{n,j,2},\ldots)^\top$ and define
\begin{equation*}
x_{n,j} = x_{n,j,1} q^{-1} + x_{n,j,2} q^{-2} + \cdots.
\end{equation*}
Then the $n$th point $\boldsymbol{x}_n$ of the sequence $\cS(C_1,\ldots,C_s)$ is given by $\boldsymbol{x}_n = (x_{n,1}, \ldots, x_{n,s})$. A sequence $\cS(C_1,\ldots,C_s)$ constructed this way is called a {\it digital sequence (over $\ZZ_q$)} with generating matrices $C_1,\ldots,C_s$.
\end{definition}

Under certain conditions on the generating matrices it can be shown that the star discrepancy of $\cS(C_1,\ldots,C_s)$ is of order of magnitude $(\log N)^s$ in $N$. For more information we refer to \cite{DP10,nie87,niesiam} and the references therein. Lower bounds for the star discrepancy of digital sequences are difficult to prove and until now we only have the general lower bound \eqref{bdroth} for arbitrary sequences, but no specific results or even improvements for digital sequences. 
There is only a single result in dimension $s=2$ due to Faure \cite{fau95} who showed for one specific $(0,2)$-sequence over $\ZZ_2$ the lower bound $$\limsup_{N \rightarrow \infty} \frac{D_N^\ast(\cS)}{(\log N)^2} \ge \frac{1}{24 (\log 2)^2}.$$

In this paper we are interested in metrical results for the star discrepancy of digital sequences. For metrical problems we need a suitable  probability measure on the set of all $s$-tuples of $\NN \times \NN$ matrices over $\ZZ_q$. We introduce this probability measure in three steps:

\begin{itemize}
\item First, let $\widetilde{\mu}$ be the measure on the sequence space $\ZZ_q^{\NN}$ induced by the equiprobability
measure on $\ZZ_q$. 
\item Then, to define a probability measure $\mu$ on $\ZZ_q^{\NN \times \NN}$, we consider the set $\ZZ_q^{\NN \times \NN}$ of all infinite matrices over $\ZZ_q$ as the product of denumerable many copies of the sequence space $\ZZ_q^{\NN}$ over $\ZZ_q$, and define $\mu$ as
the product measure induced by $\widetilde{\mu}$ on $\ZZ_q^{\NN}$. 
\item Finally, the probability measure $\mu_{s}$ on set $(\ZZ_q^{\NN \times \NN})^s$ of all $s$-tuples of $\NN \times \NN$ matrices over $\ZZ_q$ is the $s$-fold product measure induced by the probability measure $\mu$ on the set
$\ZZ_q^{\NN \times \NN}$.
\end{itemize}

\begin{remark}\rm 
Let us remark the following concerning the
measure $\widetilde{\mu}$. We can identify each $\vec{c} = (c_{1}, c_{2},
\ldots) \in \ZZ_q^{\NN}$ with its generating function
$L=\sum_{k=1}^{\infty} c_{k} z^{-k} \in \ZZ_q((z^{-1}))$, where
$\ZZ_q((z^{-1}))$ is the field of formal Laurent series over $\ZZ_q$ in
the variable $z^{-1}$. In this way we can identify $\ZZ_q^{\NN}$ with the set $H$ of all generating
functions. Consider the {\it discrete exponential valuation} $\widetilde{\nu}$ on $H$ which is defined by $\widetilde{\nu}(L)=-w$ if $L\not=0$ and $w$ is the least index with $c_w \not=0$. For $L=0$ we set $\widetilde{\nu}(0)=-{\infty}$. With the topology induced by the discrete exponential valuation $\widetilde{\nu}$ and with respect to addition, $H$ is a compact abelian group, and $\widetilde{\mu}$ then is the unique Haar probability measure on $H$.
\end{remark}

With the probability measure $\mu_s$ at hand Larcher~\cite{L98} proved the following metrical upper bound on the star discrepancy of digital sequences:

\begin{theorem}[Larcher, 1998]\label{thmLar}
Let $s \in \NN$, let $q$ be a prime number and let $\varepsilon>0$.
Then for $\mu_s$-almost all $s$-tuples $(C_1,\ldots,C_s) \in (\ZZ_q^{\NN \times \NN})^s$ of generating matrices the digital
sequence $\cS(C_1,\ldots,C_s)$ over $\ZZ_q$ has star discrepancy satisfying
$$D_N^\ast(\cS(C_1,\ldots,C_s)) \le c(q,s,\varepsilon) (\log N)^s (\log \log
N)^{2+\varepsilon}$$ with some $c(q,s,\varepsilon)>0$ not depending on $N$.
\end{theorem}

This means that $\mu_s$-almost all digital sequences achieve the conjectured best possible bound on the star discrepancy up to some $\log \log N$ term.

In this paper we will show that the result of Theorem~\ref{thmLar} is best possible (up to some $\log \log N$ term). We will prove:

\begin{theorem}\label{th2}
Let $s \in \NN$ and let $q$ be a prime number. Then for $\mu_s$-almost all $s$-tuples $(C_1,\ldots,C_s) \in (\ZZ_q^{\NN \times \NN})^s$ of generating matrices the digital
sequence $\cS(C_1,\ldots,C_s)$ over $\ZZ_q$ has star discrepancy satisfying
$$D_N^\ast(\cS(C_1,\ldots,C_s)) \ge c(q,s) (\log N)^s \log \log
N \ \ \ \mbox{ for infinitely many } N \in \NN$$ with some $c(q,s)>0$ not depending on $N$.
\end{theorem}

For the proof of the result we use an approach similar to the technique used by Beck \cite{beck} to give a metric lower bound for the discrepancy of Kronecker sequences.

\section{Auxiliary results}

We begin with some basic notations. Throughout the paper let $q$ be a prime number and let $s \in \NN$. 
Let $C,C_1,C_2,\ldots \in \ZZ_q^{\NN \times \NN}$ and let $k,\ell,k_1,k_2,\ldots \in \NN_0$.  By $\vec{k},\vec{\ell},\vec{k}_1,\vec{k}_2,\ldots$ we denote the digit vectors of $k,\ell,k_1,k_2,\ldots$ in base $q$, respectively, i.e., for $k=\kappa_0+\kappa_1 q+\cdots +\kappa_r q^r$ with $\kappa_r \not=0$ we set $$\vec{k}=(k_0,k_1,\ldots,k_r,0,\ldots)^\top \in (\ZZ_q^\NN)^\top$$ and we define $$\len(k)=\len(\vec{k}):=r= \lfloor \log_q k\rfloor.$$ For $k_1,k_2 \in \NN_0$ we write $k_1 \oplus k_2$ for the integer with digit vector $\vec{k}_1 +\vec{k}_2$, where the addition of digit vectors is always component wise modulo $q$. By $C^\top$ we denote the transpose of the matrix $C$.

An important tool in our analysis are $q$-adic Walsh functions which we introduce now:

\begin{definition}[$q$-adic Walsh functions]\rm
Let $q$ be a prime number and let $\omega_q:=\exp(2 \pi \icomp/q)$ be the $q$th root of unity. For $j \in \NN_0$ with $q$-adic expansion $j=j_0+j_1 q+j_2 q^2+\cdots$ (this expansion is obviously finite) the $j$th $q$-adic {\it Walsh function} $\walb_j:\RR \rightarrow \CC$, periodic with period one, is defined as $$\walb_j(x)=\omega_q^{j_0 \xi_1+j_1 \xi_2+j_2 \xi_3+\cdots}$$ whenever $x \in [0,1)$ has $q$-adic expansion of the form $x=\xi_1 q^{-1}+\xi_2 q^{-2}+\xi_3 q^{-3}+\cdots$ (unique in the sense that infinitely many of the digits $\xi_i$ must be different from $q-1$).
\end{definition}

In the following it is sometimes convenient to use the following notation: to $x \in [0,1)$ with $q$-adic expansion of the form $x=\xi_1 q^{-1}+\xi_2 q^{-2}+\xi_3 q^{-3}+\cdots$ we associate its $q$ adic digit vector $$\vec{x}=(\xi_1,\xi_2,\xi_3,\ldots)^\top$$ and then we write $$\walb_j(\vec{x}):=\walb_j(x).$$

For a vector $\vec{b}=(b_1,b_2,\ldots)^\top \in (\ZZ_q^\NN)^\top$ we say $\nu(\vec{b})=-w$ if it is of the form $$\vec{b}=(0,0,\ldots,0,b_w,b_{w+1},\ldots)^\top \ \mbox{ with }\ b_w \not=0.$$ 

In the following lemma we collect some useful properties of Walsh functions which we require for the proof of Theorem~\ref{th2}. More information on Walsh functions can be found in \cite[Appendix~A]{DP10}.

\begin{lemma}\label{le1}
We have 
\begin{enumerate}
\item \label{ass1} $\walb_j(C_1 \vec{k}_1 + C_2 \vec{k}_2)= \walb_{j}(C_1 \vec{k}_1) \walb_{j}(C_2 \vec{k}_2)$;
\item \label{ass2} $$\sum_{x=0}^{q^m-1} \walb_k(x/q^m) \overline{\walb_\ell(x/q^m)}=\left\{
\begin{array}{ll}
1 &  \mbox{ if } \nu(\vec{k}-\vec{\ell}) < -m, \\
0 & \mbox{ otherwise},
\end{array}\right.$$ (orthonormality of Walsh functions); and
\item \label{ass3} $$\int_{\ZZ_q^{\NN \times \NN}} \walb_i(C^\top \vec{k}) \rd \mu(C)= \left\{
\begin{array}{ll}
1 & \mbox{ if } i=0 \mbox{ or } \vec{k}=\vec{0},\\
0 & \mbox{ otherwise}. 
\end{array}\right.$$
\end{enumerate}
\end{lemma}

\begin{proof}
Point {\it \ref{ass1}.} and {\it \ref{ass2}.} are standard properties of Walsh functions and are easily deduced from their definition (see also \cite[Appendix~A]{DP10}). The assertion of {\it \ref{ass3}.} is clear whenever $i=0$ or $k=0$. Assume now that $i \not=0$ and $k\not=0$ with $q$-adic digit expansions $i=i_0+i_1 q+\cdots +i_{r-1} q^{r-1}$ and $k=k_0+k_1 q+\cdots +k_{m-1} q^{m-1}$, respectively. Then we have
$$\int_{\ZZ_q^{\NN \times \NN}} \walb_i(C^\top \vec{k}) \rd \mu(C) =  \prod_{a=0}^{r-1} \prod_{b=0}^{m-1} \sum_{c=0}^{q-1} \omega_q^{c i_a k_b}.$$ Since $i,k$ are both different from zero, there exist $a,b$ such that 
$i_a k_b \not=0$ and for this choice we have $\sum_{c=0}^{q-1} \omega_q^{c i_a k_b}=0$ and the result follows.
\end{proof}

Now we present a series of auxiliary lemmas for the proof of Theorem~\ref{th2}.\\

\begin{lemma}\label{le2}
For $m \in \NN$ and $k_1,\ldots,k_s \in \NN_0$, not all of them $0$, let $$M_m(k_1,\ldots,k_s):=\{(C_1,\ldots,C_s) \in (\ZZ_q^{\NN \times \NN})^s \ : \ \nu(C_1^\top\vec{k}_1+\cdots +C_s^\top\vec{k}_s) \le -m\}.$$ Then we have $$\mu_s( M_m(k_1,\ldots,k_s) ) = \frac{1}{q^{m-1}}.$$
\end{lemma}

\begin{proof}
Let $\chi$ be the characteristic function of the interval $[0,q^{-(m-1)})$. Then, according to \cite[Lemma~3.9]{DP10}, $\chi$ has a finite Walsh series representation in base $q$ of the form $$\chi(x)=\sum_{i=0}^{q^{m-1} -1} a_i \walb_i(x)$$ with $a_0=q^{-(m-1)}$. In the following we identify $x$ with its $q$-adic digit vector $\vec{x}$ and we write $\chi(\vec{x}):= \chi(x)$. With this notation we have
\begin{eqnarray*}
\mu( M_m(k_1,\ldots,k_s) ) & = & \int_{(\ZZ_q^{\NN \times \NN})^s} \chi(C_1^\top \vec{k}_1+\cdots +C_s^\top \vec{k}_s) \rd \mu_s(C_1,\ldots,C_s)\\
& = & a_0 + \sum_{i=1}^{q^{m-1}-1} a_i  \int_{(\ZZ_q^{\NN \times \NN})^s} \walb_i(C_1^\top\vec{k}_1+\cdots +C_s^\top \vec{k}_s) \rd \mu_s(C_1,\ldots,C_s)\\ 
& = & \frac{1}{q^{m-1}} + \sum_{i=1}^{q^{m-1}-1} a_i \prod_{j=1}^s  \int_{\ZZ_q^{\NN \times \NN}} \walb_{i}(C_j^\top \vec{k}_j) \rd \mu(C_j)\\
& = & \frac{1}{q^{m-1}},
\end{eqnarray*}
according to Lemma~\ref{le1}, since at least one of the $\vec{k}_j$ is different from $\vec{0}$.
\end{proof}

For $\bsk=(k_1,\ldots,k_s) \in \NN_0^s$ and $\bsell=(\ell_1,\ldots,\ell_s) \in \NN_0^s$ we say that $\bsk$ and $\bsell$ are strongly dependent over $\ZZ_q$ if there is a $c \in \ZZ_q\setminus\{0\}$ such that $$\vec{k}_i= c \vec{\ell}_i \ \ \mbox{ for all }\ \ i=1,\ldots,s.$$

\begin{lemma}\label{le3a}
For $\bsk=(k_1,\ldots,k_s) \in \NN_0^s$ and $\bsell=(\ell_1,\ldots,\ell_s) \in \NN_0^s$ we have $$\prod_{u=1}^s \int_{\ZZ_q^{\NN \times \NN}} \walb_i(C_u^\top \vec{k}_u) \walb_j(C_u^\top \vec{\ell}_u) \rd \mu(C_u)=1$$ if and only if  we are in one of the following cases:
\begin{enumerate}
\item $i=j=0$; or
\item $k_u=\ell_u=0$ for all $u=1,\ldots,s$; or
\item $j=0$ and $k_u=0$ for all $u=1,\ldots,s$; or
\item $i=0$ and $\ell_u=0$ for all $u=1,\ldots,s$; or
\item $\vec{j}=c \vec{i}$ and $\vec{k}_u = c \vec{\ell}_u$ for all $u=1,\ldots,s$ for some $c \in \ZZ_q\setminus\{0\}$. 
\end{enumerate}
In all other cases the above product equals 0.
\end{lemma}

\begin{proof}
The assertion is clear whenever $i=j=0$, this is item {\it 1.} Let 
\begin{eqnarray*}
\vec{k}_u & = & (k_{u,0},k_{u,1},\ldots)^\top\\
\vec{\ell}_u & = & (\ell_{u,0},\ell_{u,1},\ldots)^\top\\
\vec{i} & = & (i_0,i_1,\ldots)^\top\\
\vec{j} & = & (j_0,j_1,\ldots)^\top 
\end{eqnarray*}
be the $q$-adic digit vectors of $k_u,\ell_u, i$ and $j$, respectively. By the definition of Walsh functions it easily follows that $$\walb_i(C_u^\top \vec{k}_u) \walb_j(C_u^\top \vec{\ell}_u) = \prod_{a,b=0}^{\infty}\omega_b^{c_{u,b+1,a+1}(k_{u,a} i_b+\ell_{u,a}j_b)},$$ where $c_{u,a,b}$ is the element in the $a$th row and $b$th column of the matrix $C_u$. Note that this product is in fact a finite product since the components of the vectors $\vec{k}_u,\vec{l}_u,\vec{i},\vec{j}$ become eventually 0.

Hence the product from the lemma is 1 if and only if $$k_{u,a} i_b+\ell_{u,a} j_b=0 \ \ \ \mbox{ for all }\ u=1,\ldots ,s \ \mbox{ and }\ a,b \in \NN_0$$ and it is 0 in all other cases.  

If $i$ and $j$ not both are zero, then we may assume without loss of generality that there is a $b$ such that $i_b \not=0$. Hence $$k_{u,a}=(- i_b^{-1} j_b) \ell_{u,a} \ \ \ \mbox{ for all }\ a \in \NN_0 \ \mbox{ and } \ u=1,\ldots,s.$$ With $c:=-i_b^{-1} j_b$ we obtain $\vec{k}_u=c \vec{\ell}_u$ for all $u=1,\ldots,s$.

Further, if $i_b=0$ for some $b$, then $$\ell_{u,a} j_b=0 \ \ \ \mbox{ for all } \ \ a \in \NN_0 \ \mbox{ and }\ u=1,\ldots,s$$ and hence $j_b=0$ or $\vec{l}_u=\vec{0}$ for all $u=1,\ldots,s$. To summarize: If not all $\vec{\ell}_u$ are equal to $\vec{0}$, then $i_b=0$ implies $j_b=0$ and therefore $i_b c =j_b$ holds for all $b \in \NN_0$, i.e., $$\vec{j}=c \vec{i}.$$ This proves item {\it 5.}

If the $\vec{\ell}_u$ are all equal to $\vec{0}$, then the product reduces to $$\prod_{u=1}^s \int_{\ZZ_q^{\NN \times \NN}} \walb_i(C_u^\top \vec{k}_u) \rd \mu(C_u)$$ and this is equal to one, iff $i=0$ (item {\it 4.}) or $\vec{k}_u=\vec{0}$ for all $u=1,\ldots,s$ (item {\it 2.}). Item {\it 3.} follows in the same way as item {\it 4.}
\end{proof}


\begin{lemma}\label{le3b}
Let $\overline{P} \subseteq \NN^s$. For $(k_1,\ldots,k_s) \in \overline{P}$ with $\len(k_j)=r_j$ for $j=1,\ldots,s$ let $\beta_1(k_1),\ldots,\beta_s(k_s) \in \NN_0$ with $\beta_j(k_j)=0$ or $\len(\beta_j(k_j))> \len(k_j)$ for all $j=1,\ldots ,s$ but not all of them equal to zero. Let \begin{eqnarray*}
\widetilde{M} & := & \{ (C_1,\ldots,C_s)\in (\ZZ_q^{\NN \times \NN})^s \ : \ \nu(C_1^\top \vec{k}_1+\cdots +C_s^\top \vec{k}_s) \le -(r_1+\cdots +r_s), \\
& & \hspace{0.5cm} \nu(C_1^\top(\vec{k}_1+\vec{\beta}_1(k_1))+\cdots +C_s^\top(\vec{k}_s+\vec{\beta}_s(k_s))) \le - \lfloor (r_1+\cdots +r_s)/2 \rfloor\\
& & \hspace{0.5cm} \mbox{for infinitely many } \ (k_1,\ldots,k_s) \in \overline{P}\}.
\end{eqnarray*}
Then we have $$ \mu_s(\widetilde{M})=0.$$
\end{lemma}

\begin{proof}
For given $k_1,\ldots,k_s \in \NN$ not all $k_j=1$ let 
$$M_1(k_1,\ldots,k_s) := \{(C_1,\ldots,C_s)\, : \, \nu(C_1^\top \vec{k}_1+\cdots +C_s^\top \vec{k}_s) \le - (r_1+\cdots +r_s)\}$$ and 
\begin{eqnarray*}
M_2(k_1,\ldots,k_s) & := & \bigg\{(C_1,\ldots,C_s)\, : \\
&& \hspace{0.5cm} \nu(C_1^\top (\vec{k}_1+\vec{\beta}_1(k_1))+\cdots +C_s^\top(\vec{k}_s+\vec{\beta}_s(k_s))) \le - \left\lfloor \frac{r_1+\cdots +r_s}{2}\right\rfloor\bigg\}.
\end{eqnarray*}  
We use the notation $m:=r_1+\cdots+r_s$. Let $\chi_1$ be the characteristic function of the interval $[0,q^{-(m-1)})$ and let $\chi_2$ be the characteristic function of the interval $[0,q^{-(\lfloor m/2 \rfloor-1)})$. Again by \cite[Lemma~3.9]{DP10} we have finite Walsh series representations for $\chi_1$ and $\chi_2$ in base $q$ given by
$$\chi_1=\sum_{i=0}^{q^{m-1}-1} a_i^{(1)} \ \walb_i \ \ \mbox{ and }\ \ \chi_2=\sum_{i=0}^{q^{\lfloor m/2 \rfloor -1} -1} a_i^{(2)}\  \walb_i$$ with $$a_0^{(1)}=\frac{1}{q^{m-1}} \ \ \ \mbox{ and }\ \ \ a_0^{(2)}=\frac{1}{q^{\left\lfloor m/2\right\rfloor -1}}.$$ Now we have
\begin{eqnarray*}
\lefteqn{\mu_s( M_1(k_1,\ldots,k_s) \cap M_2(k_1,\ldots,k_s) )}\\
& = & \int_{(\ZZ_q^{\NN \times \NN})^s} \chi_1(C_1^\top \vec{k}_1+\cdots +C_s^\top \vec{k}_s) \\
&& \hspace{4cm}\times \chi_2(C_1^\top(\vec{k}_1+\vec{\beta}_1(k_1))+\cdots +C_s^\top(\vec{k}_s+\vec{\beta}_s(k_s))) \rd \mu_s(C_1\ldots,C_s)\\
& = & a_0^{(1)} a_0^{(2)} + \sum_{i,j \atop (i,j)\not=(0,0)} a_i^{(1)} a_j^{(2)} \int_{(\ZZ_q^{\NN \times \NN})^s} \walb_i(C_1^\top \vec{k}_1+\cdots +C_s^\top \vec{k}_s)\\
& & \hspace{4cm} \times \walb_j(C_1^\top(\vec{k}_1+\vec{\beta}_1(k_1))+\cdots +C_s^\top(\vec{k}_s+\vec{\beta}_s(k_s))) \rd \mu_s(C_1, \ldots, C_s)\\
& = & a_0^{(1)} a_0^{(2)} + \sum_{i,j \atop (i,j)\not=(0,0)} a_i^{(1)} a_j^{(2)} \prod_{u=1}^s \int_{\ZZ_q^{\NN \times \NN}} \walb_i(C_u^\top \vec{k}_u)\walb_j(C_u^\top(\vec{k}_u+\vec{\beta}_u(k_u))) \rd \mu(C_u).
\end{eqnarray*}
Since $\len(k_u \oplus \beta_u(k_u)) > \len(k_u)$ for at least one $u$ it follows that $\vec{k}_u$ and $\vec{k}_u+\vec{\beta}_u(k_u)$ cannot be strongly dependent for all $u$. Hence it follows from Lemma~\ref{le3a} that $$\prod_{u=1}^s \int_{\ZZ_q^{\NN \times \NN}} \walb_i(C_u^\top \vec{k}_u)\walb_j(C_u(\vec{k}_u+\vec{\beta}_u(k_u))) \rd \mu(C_u)=0.$$
So $$\mu_s( M_1(k_1,\ldots,k_s) \cap M_2(k_1,\ldots,k_s) )=a_0^{(1)} a_0^{(2)}= \frac{1}{q^{m+ \lfloor m/2 \rfloor -2}}$$ and $$\mu_s(\widetilde{M}) \le \lim_{R \rightarrow \infty} \sum_{r_1,\ldots,r_s \atop r_1+\cdots+r_s \ge R} \sum_{k_1,\ldots,k_s \atop \len(k_i)=r_i} \frac{1}{q^{r_1+\cdots+r_s+ \lfloor (r_1+\cdots+r_s)/2 \rfloor -2}}=0.$$
\end{proof}

\begin{lemma}\label{le5}
Let $(\Omega, \mathcal{A},\mu)$ be a measure space and let $(A_n)_{n \ge 1}$ be a sequence of sets $A_n \in \mathcal{A}$ such that $$\sum_{n=1}^\infty \mu(A_n)=\infty.$$ Then the set $A$ of points falling in infinitely many sets $A_n$ is of measure $$\mu(A) \ge \limsup_{Q \rightarrow \infty} \frac{\left(\sum_{n=1}^Q \mu(A_n)\right)^2}{\sum_{n,m=1}^Q \mu(A_n \cap A_m)}.$$ 
\end{lemma}

\begin{proof}
This is \cite[Lemma~5 in Chapter I]{Sp}. A proof can be found there. 
\end{proof}

\begin{lemma}\label{le4}
Let $P \subseteq \NN^s$ such that it does not contain any two strongly dependent elements $\bsk$ and $\bsell$. Let 
\begin{eqnarray*}
\overline{M}=\{(C_1,\ldots,C_s)\in (\ZZ_q^{\NN \times \NN})^s & : & \nu(C_1^\top \vec{k}_1+\cdots +C_s^\top \vec{k}_s) \le -F(r_1,\ldots,r_s)\\
& & \mbox{ for infinitely many } (k_1,\ldots,k_s)\in P\},
\end{eqnarray*}
where $r_i=\len(k_i)$, and $F: \NN_0^s \rightarrow \NN$ is such that $$\sum_{(k_1,\ldots,k_s) \in P} \frac{1}{q^{F(r_1,\ldots,r_s)}}=\infty.$$ Then $$\mu_s(\overline{M})=1.$$
\end{lemma}

\begin{proof}
For given  $(k_1,\ldots,k_s) \in P$ let $$M(k_1,\ldots,k_s):=\{(C_1,\ldots,C_s) \in (\ZZ_q^{\NN \times \NN})^s \, : \, \nu(C_1^\top \vec{k}_1+\cdots+C_s^\top \vec{k}_s) \le -F(r_1,\ldots,r_s)\}.$$  With the same proof as for Lemma~\ref{le2} we have $$\mu_s(M(k_1,\ldots,k_s))=\frac{1}{q^{F(r_1,\ldots,r_s)-1}}$$ and hence $$\sum_{(k_1,\ldots,k_s) \in P} \mu_s(M(k_1,\ldots,k_s)) = q \sum_{(k_1,\ldots,k_s) \in P} \frac{1}{q^{F(r_1,\ldots,r_s)}}=\infty.$$ Now we can use Lemma~\ref{le5} to obtain 
\begin{equation}\label{ssp1}
\mu_s(\overline{M}) \ge \lim_{R \rightarrow \infty} \frac{\left( \sum\limits_{(k_1,\ldots,k_s)\in P\atop r_1+\cdots +r_s \le R} \mu_s(M(k_1,\ldots,k_s))\right)^2}{\sum\limits_{(k_1,\ldots,k_s)\in P \atop {(\ell_1,\ldots ,\ell_s)\in P \atop \sum \len(k_j), \sum \len(\ell_j) \le R}} \mu_s(M(k_1,\ldots,k_s) \cap M(\ell_1,\ldots,\ell_s))}.
\end{equation}

Proceeding in the same way as in the proof of Lemma~\ref{le3b} we obtain $$\mu_s(M(k_1,\ldots,k_s) \cap M(\ell_1,\ldots, \ell_s))=\mu_s(M(k_1,\ldots,k_s)) \mu_s(M(\ell_1,\ldots,\ell_s))$$ provided that $(k_1,\ldots,k_s)$ and $(\ell_1,\ldots,\ell_s)$ are not strongly dependent. But this cannot happen by the definition of the set $P$.

If we denote the summands of the sum in the denominator of \eqref{ssp1} in any order by $a_1,a_2,\ldots,a_Q$, then the expression on the right hand side of \eqref{ssp1} can be written as 
\begin{equation}\label{lim1}
\lim_{Q \rightarrow \infty} \frac{\left(\sum_{k=1}^Q a_k\right)^2}{\left(\sum_{k=1}^Q a_k\right)^2+\sum_{k=1}^Q a_k - \sum_{k=1}^Q a_k^2}.
\end{equation}
Since $0 \le a_k \le 1$ for all $k$, and since $\lim_{Q \rightarrow \infty} \sum_{k=1}^Qa_k=\infty$ the limit in \eqref{lim1} is one and the result follows.
\end{proof}

We need some notation. For $m \in \NN_0$ let 
\begin{eqnarray*}
\QQ(q^m) & = & \{x=r q^{-m} \in [0,1) \, : \, r=0,\ldots,q^m -1\},\\
\QQ^s(q^m) & = & \{\bsx=(x_1,\ldots,x_s) \in [0,1)^s \, : x_j \in \QQ(q^m) \ \mbox{ for }\  j=1,\ldots,s\}.
\end{eqnarray*}

\begin{lemma}\label{le6}
Let $\cS(C_1,\ldots,C_s)$ be a digital sequence over $\ZZ_q$ with generating matrices $C_1,\ldots,C_s$. Let $N \in \NN$ with $q$-adic expansion $N=N_{m-1}q^{m-1}+\cdots +N_1 q + N_0$ and let $\bsx=(x_1,\ldots,x_s) \in \QQ^s(q^m)$.  Then we have
\begin{eqnarray*}
D(\bsx,N) =  \sum_{k_1,\ldots, k_s=0\atop (k_1,\ldots,k_s)\not=(0,\ldots ,0)}^{q^m-1} \left(\prod_{j=1}^s J_{k_j}(x_j)\right) G(N,\vec{b}(k_1,\ldots,k_s)),  
\end{eqnarray*}
where for $k=\kappa q^{a-1}+k'$ with $a \in \NN$, $1 \le
\kappa <q$ and $0 \le k' < q^{a-1}$ we have 
\begin{eqnarray}\label{valJk}
J_k(x) & = & \frac{1}{q^a} \Bigg(\frac{1}{1-\omega_q^{-\kappa}}\
\overline{\walb_{k'}(x)} + \left(\frac{1}{2} +
\frac{1}{\omega_q^{-\kappa}-1}\right) \ \overline{\walb_k(x)} \\ &&
\qquad + \sum_{c=1}^{m-a} \sum_{l = 1}^{q-1} \frac{1}{q^c
(\omega_q^l -1)} \ \overline{\walb_{l q^{a+c-1}+k}(x)}- \frac{1}{2 q^{m-a}}\ \overline{\walb_k(x)} \Bigg)
\end{eqnarray}
and for $k = 0$ we have
\begin{equation}\label{valJ0}
J_0(x) = \frac{1}{2} + \sum_{c=1}^m \sum_{l=1}^{q-1}
\frac{1}{q^c (\omega_q^l-1)} \ \overline{\walb_{l q^{c-1}}(x)}-\frac{1}{2 q^m},\nonumber
\end{equation} 
and where $$G(N,\vec{b}(k_1,\ldots,k_s))=\left[\omega_q^{b_{w+1} N_{w+1}+\cdots +b_{m-1} N_{m-1}} q^w \left(\frac{\omega_q^{b_w N_w}-1}{\omega_q^{b_w}-1} + \omega_q^{b_w N_w} \left\{\frac{N}{q^w}\right\}\right)\right]$$ with $$\vec{b}=\vec{b}(k_1,\ldots,k_s)=(b_0,b_1,b_2,\ldots)^\top =\sum_{j=1}^s C_j^{\top} \vec{k}_j$$ and $w =-\nu(\vec{b})$. However, if $w \ge m$, then we put $G(N,\vec{b})=N$.
\end{lemma}

\begin{proof}
Denote the $n$th element of $\cS(C_1,\ldots,C_s)$ by $\bsx_n=(x_{n,1},\ldots,x_{n,s})$. We have 
\begin{eqnarray*}
D(\bsx,N) = \sum_{n=0}^{N-1} \left(\prod_{j=1}^s \chi_{[0,x_j)} (x_{n,j}) - x_1 \cdots x_s\right).
\end{eqnarray*}
Since $x_j \in \QQ(q^m)$ for all $j=1,\ldots ,s$ it follows that the characteristic functions $\chi_{[0,x_j)}(x)$ have a finite Walsh-series representation of the form $$ \chi_{[0,x_j)}(x)= \sum_{k_j=0}^{q^m -1} J_{k_j}(x_j) \walb_{k_j}(x),$$ see \cite[Lemma~3.9]{DP10}, with Walsh coefficients $J_{k_j}(x_j)  =  \int_0^{x_j} \overline{\walb_k(t)} \rd t$. Hence we obtain
\begin{eqnarray*}
D(\bsx,N) & = &  \sum_{n=0}^{N-1} \left(\sum_{k_1,\ldots, k_s=0}^{q^m -1} \left( \prod_{j=1}^s J_{k_j}(x_j) \walb_{k_j}(x_{n,j})\right) - x_1 \cdots x_s\right)\\
& = & \sum_{n=0}^{N-1} \left(\sum_{k_1,\ldots, k_s=0\atop (k_1,\ldots,k_s)\not=(0,\ldots ,0)}^{q^m -1} \left( \prod_{j=1}^s J_{k_j}(x_j)\right) \walb_{\bsk}(\bsx_n)\right)\\
& = & \sum_{k_1,\ldots, k_s=0\atop (k_1,\ldots,k_s)\not=(0,\ldots ,0)}^{q^m -1} \left( \prod_{j=1}^s J_{k_j}(x_j)\right) \sum_{n=0}^{N-1}  \walb_{\bsk}(\bsx_n).
\end{eqnarray*}
According to \cite[Lemma~14.8]{DP10} the values of $J_{k_j}$ are of the form as given in \eqref{valJk} and \eqref{valJ0}, respectively. It remains to evaluate the sum $\sum_{n=0}^{N-1}  \walb_{\bsk}(\bsx_n)$. According to the construction of a digital sequence the $q$ adic digits of the $j$th component of the $n$th element $\bsx_n$ of the sequence are given by $\vec{x}_{n,j}=C_j \vec{n}$. Hence we have
\begin{eqnarray*}
\sum_{n=0}^{N-1}  \walb_{\bsk}(\bsx_n) & = & \sum_{n=0}^{N-1} \prod_{j=1}^s \walb_{k_j}(C_j \vec{n}_j)=\sum_{n=0}^{N-1} \prod_{j=1}^s \omega_b^{\langle \vec{k}_j,C_j \vec{n}\rangle}= \sum_{n=0}^{N-1} \omega_b^{\langle \sum_{j=1}^s C_j^{\top} \vec{k}_j,\vec{n}\rangle}\\
& = & \sum_{n=0}^{N-1} \omega_b^{\langle \vec{b},\vec{n}\rangle} = \sum_{n=0}^{N-1} \omega_q^{b_0 n_0+\cdots b_{m-1} n_{m-1}},
\end{eqnarray*}
where $\langle \cdot ,\cdot \rangle$ denotes the usual inner product and where $\vec{b}=\sum_{j=1}^s C_j^{\top} \vec{k}_j$.
Now we use the notation $N(w)=N_0+\cdots+N_w q^w$. Splitting up the last sum we obtain 
\begin{eqnarray*}
\sum_{n=0}^{N-1}  \walb_{\bsk}(\bsx_n) & = &  \sum_{n=0}^{N-1} \omega_q^{b_0 n_0+ \cdots +b_{m-1}n_{m-1}}  \\
& = &  \sum_{n=0}^{q^{w+1}(N_{w+1}+\cdots +N_{m-1}q^{m-w-2})-1}\omega_q^{n_w b_w}
\omega_q^{b_{w+1}n_{w+1}+\cdots +b_{m-1}n_{m-1}} \nonumber \\
&  & + \;\;  \sum_{n=0}^{N(w)-1}\omega_q^{n_w b_w}
\omega_q^{b_{w+1}N_{w+1}+\cdots +b_{m-1}N_{m-1}} \nonumber \\
& = & 0 + \omega_q^{b_{w+1}N_{w+1}+\cdots +b_{m-1}N_{m-1}}
\sum_{n=0}^{N(w)-1}\omega_q^{n_w b_w}. \nonumber 
\end{eqnarray*}
We study the last sum. We have 
\begin{eqnarray*}
\sum_{n=0}^{N(w)-1}\omega_q^{n_w b_w}
& = & \sum_{n=0}^{q^w-1} \omega_q^{0 b_w}+\sum_{n=q^w}^{2 q^w-1} \omega_q^{b_w}+\cdots +\sum_{n=(N_w-1)q^w}^{N_w q^w-1} \omega_q^{(N_w-1)b_w}+\sum_{n=N_wq^w}^{N(w)-1} \omega_q^{N_wb_w}\\
& = & q^w \sum_{k=0}^{N_w-1}\left(\omega_q^{b_w}\right)^k+(N(w)-N_w q^w) \omega_q^{N_wb_w}\\
& = & q^w \left(\frac{\omega_q^{b_w N_w}-1}{\omega_q^{b_w}-1} +\left(\frac{N(w)}{q^w}-N_w\right)\omega_q^{N_wb_w}\right)\\
& = & q^w\left(\frac{\omega_q^{b_w N_w}-1}{\omega_q^{b_w}-1}+\omega_q^{b_w N_w} \left\{\frac{N}{q^w}\right\}\right)
\end{eqnarray*}
and the result follows.
\end{proof}

\section{The proof of Theorem~\ref{th2}}

Let $N \in \NN$ with $q$-adic expansion $N=N_{m-1}q^{m-1}+\cdots +N_1 q + N_0$. We use the representation of $D(\bsx,N)$ given in Lemma~\ref{le6}. For any $\bsk^\ast=(k_1^\ast,\ldots,k_s^\ast) \in \NN^s$ with the property that each of the $k_i^\ast$ is of the form $$k_i^\ast=q^{a_i^\ast -1}+q^{a_i^\ast-2}+\ell_i^\ast$$ with some $a_i^\ast \ge 3$ and some $0 \le \ell_i^\ast < q^{a_i^\ast -2}$ we put
\begin{eqnarray*}
\Lambda & := & \Lambda(\bsk^\ast) := \sum_{\bsx \in \QQ^s(q^m)} D(\bsx,N) \walb_{\bsk^\ast}(\bsx)\\
& = & \sum_{k_1,\ldots,k_s =0 \atop (k_1,\ldots,k_s)\not=(0,\ldots,0)}^{q^m -1} \left[ \prod_{j=1}^s \sum_{x_j \in \QQ(q^m)} J_{k_j}(x_j) \walb_{k_j^\ast}(x_j) \right] G(N,\vec{b}(k_1,\ldots,k_s)).
\end{eqnarray*}

By the definition of the $J_k$ and by the orthonormality of Walsh functions (see Lemma~\ref{le1}) we have $$\theta(k):=\sum_{x \in \QQ(q^m)} J_k(x) \walb_{k^\ast}(x)=0$$ unless we are in one of the following three cases (with $k=\kappa q^{a-1}+k'$ and $k^\ast=q^{a^\ast -1}+(k^\ast)'$):
\begin{enumerate}
\item \label{cs1} $k$ is such that $k'=k^\ast$, i.e., $k=\kappa q^{a^\ast +c-1}+k^\ast$ for some $c \in \NN$ and $\kappa \in \{1,\ldots,q-1\}$. In this case we have  $$\theta(k)= \frac{1}{q^{a^\ast +c}} \frac{1}{1-\omega_q^{-\kappa}}.$$
\item \label{cs2} $k$ is such that $k=k^\ast$. In this case we have $$\theta(k)=\frac{1}{q^{a^\ast}}\left(\frac{1}{2}+\frac{1}{\omega_q -1}\right)-\frac{1}{2 q^m}.$$
\item \label{cs3} $k$ is such that $k=(k^\ast)'=q^{a^\ast-2}+l^\ast$. In this case we have $$\theta(k)=\frac{1}{q^{a^\ast}} \frac{1}{\omega_q-1}.$$
\end{enumerate}

We write $(k_j^\ast)'=:\widetilde{k}_j=q^{a_j^\ast -2}+\cdots$ (note that $k_j^\ast$ is uniquely determined by $\widetilde{k}_j$),
\begin{eqnarray*}
\beta_j(\widetilde{k}_j,0) & := & 0\\
\beta_j(\widetilde{k}_j,1) & := & q^{a_j^\ast -1}
\end{eqnarray*}
and for $t \in \NN_0$ and $u_j \in \{tq-t+2,\ldots,tq-t+q\}$ we put $$\beta_j(\widetilde{k}_j,u_j)=q^{a_j^\ast -1}+q^{a_j^\ast+t}(u_j-(tq-t+1)).$$ Then for $u_j \ge 2$ we have $$\widetilde{k}_j+\beta_j(\widetilde{k}_j,u_j) = (k_j^\ast)' + q^{a_j^\ast -1}+q^{a_j^\ast+t}(u_j-(tq-t+1))=k_j^\ast + q^{a_j^\ast +t} \kappa,$$ where $\kappa=u_j-(tq-t+1)$. Hence, according to Case~\ref{cs1}, we have $$|\theta(\widetilde{k}_j+\beta_j(\widetilde{k}_j,u_j)))|\le c_1(q) \frac{1}{q^{a_j^\ast +t+1}} \le c_1(q) \frac{1}{q^{a_j^\ast +u_j/q}} $$ for some $c_1(q)>0$. Similarly, $$\widetilde{k}_j+\beta_j(\widetilde{k}_j,0)) = (k_j^\ast)' $$ and hence, according to Case~\ref{cs3}, we have $$|\theta(\widetilde{k}_j+\beta_j(\widetilde{k}_j,0)))|\le c_2(q) \frac{1}{q^{a_j^\ast}}$$ for some $c_2(q)>0$, and $$\widetilde{k}_j+\beta_j(\widetilde{k}_j,1)) = (k_j^\ast)' +q^{a_j^\ast -1}=k_j^\ast$$ and hence, according to Case~\ref{cs2}, we have $$|\theta(\widetilde{k}_j+\beta_j(\widetilde{k}_j,1)))|\le c_3(q) \frac{1}{q^{a_j^\ast}}$$ for some $c_3(q)>0$. Summing up, for all $u_j \ge 0$ we have 
\begin{equation}\label{inequjtheta}
|\theta(\widetilde{k}_j+\beta_j(\widetilde{k}_j,u_j)))|\le c_4(q) \frac{1}{q^{a_j^\ast +u_j/q}}
\end{equation}
for some $c_4(q)>0$.

Now we have $$\Lambda=\sum_{u_1,\ldots,u_s \ge 0} \left[\prod_{j=1}^s \theta(\widetilde{k}_j+\beta_j(\widetilde{k}_j,u_j))\right] G(N,\vec{b}(\widetilde{k}_1+\beta_1(\widetilde{k}_1,u_1),\ldots,\widetilde{k}_s+\beta_s(\widetilde{k}_s,u_s)))$$ where the summation is over all $u_j$ with $\widetilde{k}_j+\beta_j(\widetilde{k}_j, u_j) < q^m$ for all $j=1,\ldots,s$. For any $J \in \NN$ we have
\begin{eqnarray}\label{defLambda}
|\Lambda| & \ge & \left|\left[\prod_{j=1}^s \theta(\widetilde{k}_j)\right] G(N,\vec{b}(\widetilde{k}_1,\ldots,\widetilde{k}_s))\right|\\
&&-\left|\sum_{0 \le u_1,\ldots,u_s \le J \atop (u_1,\ldots,u_s)\not=(0,\ldots,0)} \left[\prod_{j=1}^s \theta(\widetilde{k}_j+\beta_j(\widetilde{k}_j,u_j))\right] G(N,\vec{b}(\widetilde{k}_1+\beta_1(\widetilde{k}_1,u_1),\ldots,\widetilde{k}_s+\beta_s(\widetilde{k}_s,u_s)) \right| \nonumber \\
&&-\left|\sum_{u_1,\ldots,u_s \ge 0 \atop \exists j:\ u_j >J} \left[\prod_{j=1}^s \theta(\widetilde{k}_j+\beta_j(\widetilde{k}_j,u_j))\right] G(N,\vec{b}(\widetilde{k}_1+\beta_1(\widetilde{k}_1,u_1),\ldots,\widetilde{k}_s+\beta_s(\widetilde{k}_s,u_s)) \right|.\nonumber 
\end{eqnarray}

Note that $|G(N,\vec{b}(k_1,\ldots,k_s))| \le q N$ always. Therefore and using \eqref{inequjtheta}, for the last sum in \eqref{defLambda} we have 
\begin{eqnarray*}
\left| \Sigma \right| \le \frac{q N}{q^{a_1^\ast + \cdots + a_s^\ast}} \sum_{u_1,\ldots,u_s \ge 0 \atop \exists j:\ u_j >J} q^{-\frac{u_1}{q}-\cdots - \frac{u_s}{q}}\le c_5(q,s) \frac{N}{q^{a_1^\ast + \cdots + a_s^\ast}} \frac{1}{q^{J/q}},
\end{eqnarray*}
with some $c_5(q,s)>0$ depending only on $q$ and on $s$.\\

Let the function $F$ from Lemma~\ref{le4} be such that 
\begin{equation}\label{defF}
q^{F(r_1,\ldots,r_s)} = q^{r_1+\cdots+r_s} (r_1+\cdots +r_s)^s \log (r_1+\cdots+r_s)
\end{equation}
and let $P$ from Lemma~\ref{le4} be given by 
\begin{eqnarray}\label{defP}
P & = & \{(k_1,\ldots,k_s) \in \NN^s \ : \ (k_1,\ldots,k_s)\not=(1,\ldots,1),\  k_i=q^{a_i-1}+\ell_i \nonumber\\
& &\hspace{0.5cm} \mbox{ for some } a_i \in \NN \mbox{ and } 0 \le \ell_i < q^{a_i -1} \mbox{ for all } i=1,\ldots, s\}.
\end{eqnarray}
Note that in $P$ no two elements are strongly dependent since the leading digits of the $k_i$ are all 1.
With $r=\len(k_i)$ we have
\begin{eqnarray*}
\sum_{(k_1,\ldots,k_s)\in P} \frac{1}{q^{F(r_1,\ldots,r_s)}} & = & \sum_{a_1,\ldots,a_s=0\atop a_1+\cdots +a_s \not=0}^\infty \frac{1}{q^{F(a_1,\ldots,a_s)}} \sum_{\ell_1=0}^{q^{a_1}-1}\ldots \sum_{\ell_s=0}^{q^{a_s}-1} 1\\
& = &  \sum_{a_1,\ldots,a_s=0\atop a_1+\cdots +a_s \not=0}^\infty \frac{1}{q^{F(a_1,\ldots,a_s)}} \prod_{j=1}^s \frac{q^{a_j+1}-1}{q-1}\\
& \ge & \frac{1}{(q-1)^s} \sum_{a_1,\ldots,a_s=0\atop a_1+\cdots +a_s \not=0}^\infty \frac{1}{(a_1+\cdots+a_s)^s \log (a_1+\cdots+a_s)}\\
& = & \frac{1}{(q-1)^s} \sum_{d=1}^\infty \frac{1}{d^s \log d} {s+d-1 \choose d}\\
& \ge & \frac{1}{(q-1)^s} \frac{1}{(s-1)!} \sum_{d=1}^\infty \frac{1}{d \log d}\\
& = & \infty,
\end{eqnarray*}
where we used that ${s+d-1 \choose d} \ge \frac{d^{s-1}}{(s-1)!}$.

Now we use Lemma~\ref{le4} and find that the set $\overline{M}$ for our choice of $F$ as in \eqref{defF} and $P$ as in \eqref{defP} has measure $\mu_s(\overline{M})=1$.

Next we consider the finite collection of $s$-tuples $$(\beta_1(k_1,u_1),\ldots,\beta_s(k_s,u_s))$$ for $u_1,\ldots,u_s=0,1,\ldots,J$ but not all equal to 0. Note that each of these $\beta_i(k_i,u_i)$ is either 0 or has $\len(\beta_i(k_i,u_i))>\len(k_i)$ for each $(k_1,\ldots,k_s)$ which is an element from $P$ defined above.

Now we use Lemma~\ref{le3b} where we choose $\overline{P}$ as $\overline{P}=P$ from \eqref{defP} and for any choice of $u_1,\ldots,u_s=0,1,\ldots,J$ but not all equal to 0, we choose the $\beta_i(k_i)$ from Lemma~\ref{le3b} as $\beta_i(k_i)=\beta_i(k_i,u_i)$. Then for the corresponding set $\widetilde{M}:=\widetilde{M}(u_1,\ldots,u_s)$ of Lemma~\ref{le3b} we have $\mu_s(\widetilde{M})=0$.

We set $$M:= \overline{M} \setminus \bigcup_{u_1,\ldots,u_s=0 \atop (u_1,\ldots,u_s)\not=(0,\ldots,0)}^J \widetilde{M}(u_1,\ldots,u_s)$$ and find that $\mu_s(M)=1$. 

Now we make a suitable choice for $C_1,\ldots,C_s$ and for $\bsk^\ast$.  Let $(C_1,\ldots,C_s) \in M$ and let $(\widetilde{k}_1,\ldots,\widetilde{k}_s) \in P$ be such that
$$\nu(C_1^\top \vec{\widetilde{k}}_1+\cdots +C_s^\top \vec{\widetilde{k}}_s) \le -F(r_1,\ldots,r_s) \le -(r_1+\cdots+r_s)$$ and $$\nu(C_1^\top(\vec{\widetilde{k}}_1+\vec{\beta}_1(\widetilde{k}_1,u_1))+\cdots + C_s^\top(\vec{\widetilde{k}}_s+\vec{\beta}_s(\widetilde{k}_s,u_s))) \ge -\frac{r_1+\cdots+r_s}{2},$$ where $r_i=\len(\widetilde{k}_i)$ and $\widetilde{k}_i=q^{\widetilde{a_i}-1}+\widetilde{\ell}_i$. By the definition of $M$ there are infinitely many such $s$-tuples $(\widetilde{k}_1,\ldots,\widetilde{k}_s)$.

Let $m:=\lfloor F(r_1,\ldots,r_s)\rfloor$ and $N=q^{m-1}$. We analyze the first summand in \eqref{defLambda}: Let  $$C_1^\top\vec{\widetilde{k}}_1+\cdots+C_s^\top \vec{\widetilde{k}}_s=(0,\ldots,0,b_w,b_{w+1},\ldots)^\top$$ with $b_w \not=0$. Then we have $$-w=\nu((0,\ldots,0,b_w,b_{w+1},\ldots)^\top)=\nu(C_1^\top \vec{\widetilde{k}}_1+\cdots+C_s^\top\vec{\widetilde{k}}_s) \le -F(r_1,\ldots,r_s) \le -m$$ and hence $w \ge m$. This means that $G(N,\vec{b}(\widetilde{k}_1,\ldots,\widetilde{k}_s))=N$ and hence we obtain
$$\left|\left[\prod_{j=1}^s \theta(\widetilde{k}_j)\right] G(N,\vec{b}(\widetilde{k}_1,\ldots,\widetilde{k}_s))\right| \ge c_6(q,s) \frac{N}{q^{\widetilde{a}_1+\cdots +\widetilde{a}_s}}.$$ 

Now we turn to the second  summand in \eqref{defLambda}. Let $$ C_1^\top(\vec{\widetilde{k}}_1+\vec{\beta}_1(\widetilde{k}_1,u_1))+\cdots + C_s^\top(\vec{\widetilde{k}}_s+\vec{\beta}_s(\widetilde{k}_s,u_s))=(0,\ldots,0,b_w,b_{w+1},\ldots)^\top$$ with $b_w \not=0$. Then we have 
\begin{eqnarray*}
-w & = & \nu((0,\ldots,0,b_w,b_{w+1},\ldots)^\top)\\
& = & \nu(C_1^\top(\vec{\widetilde{k}}_1+\vec{\beta}_1(\widetilde{k}_1,u_1))+\cdots + C_s^\top(\vec{\widetilde{k}}_s+\vec{\beta}_s(\widetilde{k}_s,u_s)))\\
& \ge & -\frac{r_1+\cdots +r_s}{2}
\end{eqnarray*}
and hence $$w \le \frac{r_1+\cdots +r_s}{2}.$$ This means that 
\begin{eqnarray*}
|G(N,\vec{b}(\widetilde{k}_1+\beta_1(\widetilde{k}_1,u_1),\ldots,\widetilde{k}_s+\beta_s(\widetilde{k}_s,u_s))|
& \le & c_7(q) q^w \\
& \le & c_7(q) q^{\frac{r_1+\cdots +r_s}{2}}\\
& = & c_8(q,s) q^{\frac{\widetilde{a}_1+\cdots+\widetilde{a}_s}{2}}
\end{eqnarray*}
and hence we obtain 
\begin{eqnarray*}
\left|\sum_{u_1,\ldots,u_s \ge 0 \atop \exists j:\ u_j >J} \left[\prod_{j=1}^s \theta(\widetilde{k}_j+\beta_j(\widetilde{k}_j,u_j))\right] G(N,\vec{b}(\widetilde{k}_1+\beta_1(\widetilde{k}_1,u_1),\ldots,\widetilde{k}_s+\beta_s(\widetilde{k}_s,u_s)) \right| \\ \le c_9(q,s) \frac{q^{\frac{\widetilde{a}_1+\cdots+\widetilde{a}_s}{2}}}{q^{\widetilde{a}_1+\cdots+\widetilde{a}_s}}.
\end{eqnarray*}

Altogether we have 
\begin{eqnarray*}
|\Lambda| & \ge &  c_6(q,s) \frac{N}{q^{\widetilde{a}_1+\cdots +\widetilde{a}_s}} - c_9(q,s) \frac{q^{\frac{\widetilde{a}_1+\cdots+\widetilde{a}_s}{2}}}{q^{\widetilde{a}_1+\cdots+\widetilde{a}_s}}-c_5(q,s) \frac{N}{q^{a_1^\ast + \cdots + a_s^\ast}} \frac{1}{q^{J/q}}\\
& \ge & c_{10}(q,s) \frac{N}{q^{\widetilde{a}_1+\cdots +\widetilde{a}_s}}
\end{eqnarray*}
for $J$ large enough and for $\len(\widetilde{k}_1)+\cdots +\len(\widetilde{k}_s)$ large enough. Now
\begin{eqnarray*}
\frac{N}{q^{\widetilde{a}_1+\cdots +\widetilde{a}_s}} & \ge & c_{11}(q,s) q^{F(r_1,\ldots,r_s)-r_1-\cdots -r_s} \\
& = & c_{11}(q,s)(r_1+\cdots+r_s) \log(r_1+\cdots+r_s) \\
& \ge & c_{12}(q,s) (\log N)^s \log \log N.
\end{eqnarray*}

From the definition of $\Lambda$ it follows that there exists an $\bsx \in \QQ^s(q^m) \subseteq[0,1)^s$ such that $$|D(\bsx,N)| \ge  c_{13}(q,s) (\log N)^s \log \log N$$ and the proof of Theorem~\ref{th2} is finished. \hfill $\qed$

\begin{small}
\noindent\textbf{Authors' address:}\\
\noindent Gerhard Larcher, Friedrich Pillichshammer\\
Institut f\"{u}r Finanzmathematik, Universit\"{a}t Linz, Altenbergerstr.~69, 4040 Linz, Austria\\
E-mail: \\ \texttt{gerhard.larcher@jku.at},\\ 
\texttt{friedrich.pillichshammer@jku.at}
\end{small}

\end{document}